\documentclass{article}

\usepackage{arxiv}

\usepackage[utf8]{inputenc} 
\usepackage[T1]{fontenc}    
\usepackage{hyperref}       
\usepackage{url}            
\usepackage{booktabs}       
\usepackage{amsfonts}       
\usepackage{nicefrac}       
\usepackage{microtype}      
\usepackage{lipsum}
\usepackage{amsmath}

\usepackage{amsfonts,amssymb}
\usepackage{graphicx}
\usepackage{amsthm}
\usepackage{xcolor}

\newcommand{\R}{\mathbb{R}}

\newtheorem{theorem}{Theorem}[section]
\newtheorem{proposition}{Proposition}[section]
\newtheorem{lemma}{Lemma}[section]
\newtheorem{corollary}{Corollary}[section]

\usepackage{cite}

\usepackage{authblk}

\newcommand{\p}{\partial}
\newcommand{\bb}{\begin{equation}}
\newcommand{\ee}{\end{equation}}
\newcommand{\ba}{\begin{array}}
\newcommand{\ea}{\end{array}}
\newcommand{\f}{\frac}
\usepackage[all]{xy}
\newcommand{\ds}{\displaystyle}

\newcommand{\al}{\alpha}
\newcommand{\be}{\beta}

\newcommand{\sn}{\text{sn}\,}
\newcommand{\dn}{\text{dn}\,}
\numberwithin{equation}{subsection}

\title{Invariants and wave breaking analysis of a Camassa-Holm type equation with quadratic and cubic nonlinearities}

\author{Igor Leite Freire$^1$, Nazime Sales Filho$^{2,3}$, Ligia Corrêa de Souza$^{2,4}$ and Carlos Eduardo Toffoli$^{2,4}$\\
$^1$Centro de Matem\'atica, Computa\c{c}\~ao e Cogni\c c\~ao,\\ Universidade Federal do ABC,\\ Avenida dos Estados, $5001$, Bairro Bangu,
$09.210-580$, Santo Andr\'e, SP - Brasil\\
\texttt{igor.freire@ufabc.edu.br and igor.leite.freire@gmail.com}\\
 
$^2$Programa de Pós-Graduação em Matemática, \\
 Universidade Federal do ABC, \\Avenida dos Estados, $5001$, Bairro Bangu,
$09.210-580$, Santo Andr\'e, SP - Brasil\\
$^3$Faculdade de Engenharia,\\Universidade Federal de Mato Grosso,\\Av. Fernando Corrêa da Costa, nº 2367 - Bairro Boa Esperança. Cuiabá - MT - 78060-900\\
\texttt{sales.nazime@gmail.com}\\
$^4$Instituto Federal de Educação, Ciência e Tecnologia de São Paulo,\\ Campus de Campos do Jordão,\\ Rua Monsenhor José Vita, 280,
12460-000, Abernéssia, Campos do Jordão, SP - Brasil\\
\texttt{li.correasouza@gmail.com and ligiacorrea@ifsp.edu.br} \\ \texttt{tofolis@gmail.com}
}

\begin{document}
\maketitle

\begin{abstract}
A non-local evolution equation of the Camassa-Holm type with dissipation is considered. The local well-posedness of the solutions of the Cauchy problem involving the equation is established via Kato's approach and the wave breaking scenario is also described. To prove such result, we firstly construct conserved currents for the equation and from them, conserved quantities.
\end{abstract}

\newpage

\section{Introduction}

The equation
$$
u_t-u_{txx}+3uu_x=2u_xu_{xx}+uu_{xxx},
$$
where $u=u(t,x)$, was derived by Camassa and Holm in their impacting work \cite{chprl} and is named after them. Since then such equation has been widely investigated from different perspectives, ranging from the structural viewpoint, such as symmetries and conservation laws \cite{chprl,anco,clark,pri-aims,igor-cm}, to qualitative standpoint, such as local and global well-posedness \cite{blanco,const1998-1,const2000-1,escher}, blow-up of solutions \cite{const1998-1,const1998-2} and their stability \cite{const2000-2,const2002}. While in the references \cite{const1998-1,const1998-2,const2000-1,escher} the Cauchy problem involving the Camassa-Holm equation is considered over the real line $\R$ with initial data $u_0\in H^3(\R)$, in \cite{const1998-3} is considered the question of well-posedness and blow-up of solutions of the periodic Camassa-Holm equation.

It is worth mentioning that the work by Camassa and Holm was the cornerstone of an active field in Mathematics and Mathematical Physics and since its discovery several equations of the type $u_t-u_{txx}=F(u,u_x,u_{xx},u_{xxx})$, where $F$ is a smooth function of its arguments, have been proposed and investigated, see \cite{anco,pri-book,pri-aims,raspa,silvajde2019,silva2019,igor-jde,igor-cm,hak,liu2011,mustafa} and references thereof.

The equation
\bb\label{1.0.1}
u_t-u_{txx}+3u^2u_x=uu_{xxx}+2u_xu_{xx}
\ee
was considered in \cite{waz} and has been attracting some attention, see \cite{chen2020,daros,kelly,natali,yin2010} and references thereof. We were driven to this equation by the works of Darós \cite{daros}, Darós and Arruda \cite{kelly}, and Martins and Natali \cite{natali}. 

The Cauchy problem involving \eqref{1.0.1} with a given initial data $u_0\in H^s(\R)$ is locally well-posed for $s>3/2$ . This can be proved using Kato's approach \cite{kato}, see \cite{hak,mustafa,liu2011,silvajde2019}, where the Cauchy problem of equations including \eqref{1.0.1} was considered. We would also like to point out that \eqref{1.0.1} is a particular case of generalisations of the Camassa-Holm equation previously considered in \cite{hak,mustafa,liu2011, waz}.

In this paper we are concerned with the Cauchy problem
\bb\label{1.0.2}
\left\{
\ba{l}
u_t-u_{txx}+\lambda(u-u_{xx})+3u^2u_x=uu_{xxx}+2u_xu_{xx},\\
\\
u(0,x)=u_0(x),
\ea
\right.
\ee
where $\lambda$ is a constant. This is a natural generalisation of \eqref{1.0.1} and, for $\lambda>0$, it can be seen as a (weakly) dissipative form of the equation \eqref{1.0.1}. 

\subsection{Main qualitative results}

A natural question about the Cauchy problem \eqref{1.0.2} is whether it is well-posed. A first answer to it is the following result:
\begin{theorem}\label{teo1.1}
Given $u_0\in H^{s}(\R)$, $s>3/2$, then there exist a maximal time of existence $T=T(u_0)>0$ and a unique solution $u$ to the Cauchy problem \eqref{1.0.2} satisfying the initial condition $u(0,x)=u_0(x)$, such that $u=u(\cdot,u_0)\in C^{0}([0,T);H^{s}(\R))\cap C^{1}([0,T);H^{s-1}(\R))$. Moreover, the solution depends continuously on the initial data, in the sense that the map $u_0\mapsto u(\cdot,u_0):H^{s}(\R)\rightarrow C^{0}([0,T);H^{s}(\R))\cap C^{1}([0,T);H^{s-1}(\R))$ is continuous and $T$ does not depend on $s$.
\end{theorem}

Once we have characterised conditions for having locally well-posed solutions, a follow-up question is to look for the existence or not of singularities. It is well-known that the existence of some $k>0$ such that $u_x>-k$, where $u$ is a solution of the Camassa-Holm, prevent the formation of blow-up phenomenon for that equation, see \cite{const1998-1,const1998-2,const1998-3,escher}. Our next result shows that equation \eqref{1.0.2} shares the same property with the Camassa-Holm equation.

\begin{theorem}\label{teo1.2}
Let $u$ be a solution of \eqref{1.0.2} with initial data $u(0,x)=u_0(x)$ and $m_0=u_0(x)-u_0''(x)$. Assume that $m_0\in H^1(\R)$ and there exists a positive constant $k$ such that $u_x>-k$. Then there exists a positive constant $\kappa$ such that
$\|u\|_{H^3(\R)}\leq e^{\kappa t}\,\|m_0\|_{H^1(\R)}$. In particular, $u$ does not blow-up in finite time. 
\end{theorem}

Theorem \ref{teo1.2} says that the lower limitation of the $x-$derivative is enough to avoid the development of blow-up phenomenon. However, we want to make two important observations:
\begin{itemize}
    \item Theorem \ref{teo1.2} does not give conditions to guarantee when the $x-$derivatives of the solutions of \eqref{1.0.2} are bounded from below;
    \item it does not necessarily bring much clarity about the formation of singularities.
\end{itemize}

While at first sight the comments above are somewhat frustrating, Theorem \ref{teo1.2} indicates a direction to investigate a bit more the blow-up scenario of solutions of the Cauchy problem \eqref{1.0.2}: We should look for solutions of \eqref{1.0.2} satisfying the condition
\bb\label{1.1.1}\lim\sup\limits_{t\rightarrow T}\left(\sup_{x\in\R}|u_x(t,x)|\right)=\infty,
\ee
where $T$ is the maximal time of existence assured by Theorem \ref{teo1.1}. Moreover, the same theorem, jointly with the Sobolev Embedding Theorem, implies that the local solutions of \eqref{1.0.2} have their shape bounded, that is
\bb\label{1.1.2}
\sup_{(t,x)\in[0,T)\times\R}|u(t,x)|<
\infty.
\ee

Therefore, the comments above allude that a blow-up manifestation, if it occurs, is in the form of wave breaking, as expressed by \eqref{1.1.1}. Conditions for its occurrence can now be given.

\begin{theorem}\label{teo1.3}
Given $0\not\equiv u_0 \in H^{3}(\R)$, let $u=u(t,x)$ be the corresponding solution of \eqref{1.0.2} assured by Theorem \ref{teo1.1}. Let
$$
\epsilon_0:=1-\f{2\sqrt{2}\|u_0\|_{H^1(\R)}^3+\|u_0\|_{H^1(\R)}^2}{2(u_0'(x_0))^2},
$$
$y(t):=\inf\limits_{x\in\R}u_x(t,x)$ and
$$
\lambda_0:=-\epsilon_0\f{y(0)}{4}.
$$
If 
\begin{enumerate}
    \item there exists a point $x_0\in\R$ such that
\bb\label{1.1.3}
u_0'(x_0)<-\sqrt{\sqrt{2}\|u_0 \|_{H^1(\R)}^3+\frac{1}{2}\|u_0\|_{H^1(\R)}^2},
\ee
and
\item $\lambda\in[0,\lambda_0)$,
\end{enumerate} 
then $u$ breaks at finite time bounded from above by $T_+$, where
$$
T_+=\left\{\ba{l}
\ds{\f{1}{\lambda}\ln{\left(\f{\epsilon_0 y(0)}{\epsilon_0 y(0)+4\lambda}\right)},\quad \text{if}\quad\lambda>0},\\
\\
\ds{-\f{8}{\epsilon_0 y(0)},\quad \text{if}\quad\lambda=0}.
\ea\right.
$$
\end{theorem}

We would like to observe that the condition \eqref{1.1.3} has two main implications:
\begin{itemize}
    \item Firstly, $y(0)<0$ and, therefore, $\lambda_0>0$;
    \item secondly, $\epsilon_0\in(0,1)$.
\end{itemize} 

To prove Theorem \ref{teo1.3} we follow the steps of \cite{igor-jde}, that is, we use a small parameter $\epsilon>0$ limited from above by $\epsilon_0$. Next, given the initial data $u_0$, if $\lambda\in[0,\lambda_0)$, then we can assure the wave breaking. Moreover, the restriction on the values of $\lambda$ that may lead to blow-up of solutions enables us to interpret the term $\lambda(u-u_{xx})$ in \eqref{1.0.2} as a perturbation of the equation \eqref{1.0.1}. 

We observe that the presence of the cubic and quartic non-linearities in \eqref{1.0.2} does not allow the reduction of the analysis of \eqref{1.0.2} to \eqref{1.0.1}, as it is possible for some Camassa-Holm type equations \cite{lenjde2013}. For further details about this matter, see the discussion about it in \cite{igor-jde}.

\subsection{Organisation of the paper} Theorem \ref{teo1.1}'s proof is done by showing that the Cauchy problem \eqref{1.0.2} satisfies certain conditions required by Kato \cite{kato} for proving the local existence of abstract evolution equations in Banach spaces. Then, in Section \ref{sec2} we present an overview of Kato's approach and show that the Cauchy problem \eqref{1.0.2} satisfies the conditions required by Kato's theory.

Theorems \ref{teo1.2} and \ref{teo1.3} are proved in Section \ref{sec4}. While the demonstration of Theorem \ref{teo1.2} can be done without further qualitative information of the solutions of \eqref{1.0.2}, the proof of Theorem \ref{teo1.3} requires some invariants of the solutions, namely, {\it sine qua non} ingredients to prove it in our framework are certain quantities conserved along the solutions. These quantities, better known as {\it conserved quantities}, bring a new problem to us: How can we find them? To overcome this issue, in Section \ref{sec3} we derive some invariants for the equation:
\begin{itemize}
    \item We firstly construct the invariance group of the equation, but not necessarily of the Cauchy problem. We prove that the only group of diffeomorphisms acting on $\R^3$ and leaving invariant the solutions of the equation are formed by translations in the independent variables. As a consequence, the only invariant solutions of the equation having dependence on both $t$ and $x$ are travelling waves;
    \item we construct conserved currents for the equation, which is a pair $(C^0,C^1)$ such that its divergence, when taken on the solutions of the equation, vanishes identically. As a consequence, the integral of $C^0$ over the whole domain is invariant with respect to $t$ under mild conditions on the behaviour of the solutions on the boundary (here we understand $\pm\infty$ as boundary for unbounded  domains). The invariants found in Section \ref{sec3} will be of vital importance to prove the existence of wave breaking of solutions for \eqref{1.0.2}.
\end{itemize}

From the conserved currents found and the fact that the most general solution of \eqref{1.0.2} preserving the invariance group are the travelling waves, we derive some quantities used in \cite{kelly,natali} for proving their results regarding the (in)stability of solutions.

We observe that while the steps used to prove theorems \ref{teo1.1}--\ref{teo1.3} follow the ideas presented in \cite{const1998-1,const1998-2,const1998-3,escher}, our way to construct the ingredients needed to prove Theorem \ref{teo1.3} is rather different of the usual. Actually, in order to construct invariants required to establish qualitative information for the solutions, we borrow from group-analysis \cite{bk,chev-2007,chev-2010-1,chev-2010-2,chev-2014,chev-2017,ibrabook,olverbook} the machinery to construct conserved currents and, therefore, the conserved quantities emerge as an extremely natural consequence of them.

\section{Existence and uniqueness of solutions}\label{sec2}

Here we deal with the problem of existence and uniqueness of solutions of \eqref{1.0.2} at the local level. Our main tool for establishing such result is Kato's theory \cite{kato}. To accomplish our goal, we recall some known results from the literature and next we prove the local well-posedness. Before, we introduce the notation and basic concepts used in our work.

\subsection{Notation and conventions} 

If $X$ and $H$ are a Banach and Hilbert spaces, their corresponding norm and inner product are denoted, respectively, by $\|\cdot\|_X$ and $\langle \cdot,\cdot\rangle_{H}$. We denote the usual (Hilbert) Sobolev space by $H^s(\mathbb{\R})$, for each $s\in\R$. Given two functions $f$ and $g$, their convolution is denoted by $f\ast g$. If $u=u(t,x)$, we denote by $u_0(x)$ the function $x\mapsto u(0,x)$, $m=u-u_{xx}$ and $m_0=u_0-u_0''$. Note that $m=(1-\p_x^2)u$ and then $u=g\ast m$, where $g(x)=e^{-|x|}/2$. For each $s\in\R$, $\p_x\in {\cal L}(H^{s}(\R), H^{s-1}(\R))$, where $u\mapsto u_x$, that is, $\p_x$ is a linear and bounded (and continuous) operator from $H^s(\R)$ into $H^{s-1}(\R)$. Moreover, if $s$ and $t$ are real numbers such that $s\geq t$, then ${\cal S}(\R)\subseteq H^{s}(\R)\subseteq H^t(\R)\subseteq{\cal S}'(\R)$, where ${\cal S}(\R)$ denotes the Schwartz space while ${\cal S}'(\R)$ is its topological dual space. We, indeed, shall use the estimates $\Vert\partial_x f \Vert_{H^{s-1}(\R)}\leq \Vert f\Vert_{H^{s}(\R)}$, 
$
\Vert \Lambda^{-2}f \Vert_{H^s(\R)}\leq \Vert f\Vert_{H^{s-2}(\R)}
$
and 
$\Vert \partial_x\Lambda^{-2}f \Vert_{H^s(\R)}\leq \Vert f\Vert_{H^{s-1}(\R)}
$.

\subsection{Auxiliary results}

We firstly recall that the embedding $H^{s}(\R)\hookrightarrow H^t(\R)$ is continuous and dense when $s\geq t$. A very useful result is:
\begin{lemma}\textsc{[Sobolev Embedding Theorem]}\label{lema2.4}
If $s>1/2$ and $u\in H^{s}(\R)$, then $u$ is bounded and continuous.
\end{lemma}
\begin{proof}
See \cite{linares,taylor}, pages 47 and 317, respectively.
\end{proof}

The Sobolev Embedding Theorem assures that $\|u\|_{L^\infty(\R)}\leq c\|u\|_{H^s(\R)}$, for any $s>1/2$.

\begin{lemma}\label{lema2.2}
Let $F$ and $f$ be functions such that $F\in C^\infty(\R)$, $F(0)=0$, and $f\in H^{s}(\R)$, with $s>1/2$. Then $F(f)\in H^s(\R)$.
\end{lemma}

\begin{proof}
See \cite{const-mol}, proof of Lemma 1 and references thereof.
\end{proof}

\begin{lemma}\label{lema2.3}
Let $F$ and $f$ be functions such that $F\in C^\infty(\R)$, $F(0)=0$, and $f\in H^{s}(\R)$, with $s>1/2$. If, for some $t>1/2$, the $H^t(\R)$ norm of $f$ is bounded from above by $R>0$, that is $\|f\|_{H^t(\R)}\leq R$, then there exists a positive constant $c$ depending only on $R$ such that $F(f)\in H^s(\R)$ and $\|F(f)\|_{H^s(\R)}\leq c\|f\|_{H^s(\R)}$.
\end{lemma}
\begin{proof}
See \cite{silva2019}, Lemma 2.1, or \cite{igor-jde}, claim in the Theorem 4.1.
\end{proof}

\begin{lemma}\label{lema2.4}\textsc{[Algebra property]}
For $s>1/2$, there is a constant $c>0$ such that $\Vert fg \Vert_{H^s(\R)}\leq c\Vert f\Vert_{H^s(\R)}\Vert g\Vert_{H^s(\R)}$.
\end{lemma}

\begin{proof}
See \cite{linares} or \cite{taylor}, page 51 or exercise 6 on page 320, respectively.
\end{proof}

\begin{lemma}\label{lema2.5}\textsc{(Negative Multiplier)} If $s>1/2$, then
$
\Vert fg \Vert_{H^{s-1}}\leq c\Vert f\Vert_{H^{s}}\Vert g \Vert_{H^{s-1}}
$, for some $c>0$.
\end{lemma}
\begin{proof}
See Lemma A1 in \cite{kato}.
\end{proof}

\subsection{Technical results}

In what follows we assume that $s>3/2$. Let $B:=B(0;R)\subseteq H^{s}(\R)$ be a ball (which we may assume to be closed) with centre at $0$ and radius $R>0$. Since $H^s(\R)$ is continuously and densely embedded in $H^{s-1}(\R)$ we have the inequality $\|u\|_{H^{s-1}(\R)}\leq c\|u\|_{H^s(\R)}$, for some $c>0$, that will play great importance for us. Throughout this subsection $c_{s,R}$ stands for a generic constant, eventually depending on $R$ and $s$.

\begin{proposition}\label{prop2.1}
The equation 
\bb\label{2.3.1}
u_t-u_{txx}+\lambda(u-u_{xx})+3u^2u_x=uu_{xxx}+2u_xu_{xx}
\ee
is equivalent to
\bb\label{2.3.2}
u_t+uu_x=F(u),
\ee
where 
\bb\label{2.3.3}
F(u)=-\p_x\Lambda^{-2}\left(u^2+\f{u_x^2}{2}\right)-\p_x\Lambda^{-2}\left(h(u)\right)-\lambda u
\ee
and
\bb\label{2.3.4}
h(u):=u^3-\f{3}{2}u^2.
\ee
\end{proposition}
\begin{proof}
Applying the operator $\Lambda^2$ into \eqref{2.3.2} and taking \eqref{2.3.3} and \eqref{2.3.4} into account, we obtain the equation in \eqref{2.3.1}.
\end{proof}

\subsection{Local well-posedness}\label{subsec2.2}

\begin{proposition}\label{prop2.1}
Let $F:B\rightarrow H^{s}(\R)$ be the function given by \eqref{2.3.3}. Then $F$ is Lipschitz. In particular, $F$ is bounded in $B$.
\end{proposition}
\begin{proof}
We firstly show that $F$ is well-defined. Since $u\in H^{s}$, $s>3/2$, it follows that $u_x^2\in H^{s-1}$ and $\p_x\Lambda^{-2}(u_x^2)\in H^{s}(\R)$, whereas $\p_x\Lambda^{-2}(u^2)\in H^{s+1}(\R)\subseteq H^{s}(\R)$. By Lemma \ref{lema2.2} we see that $h(u)$, given in \eqref{2.3.4}, belongs to $H^s(\R)$ provided that $u\in H^s(\R)$. Therefore, $\p_x\Lambda^{-2}h(u)\in H^{s+1}(\R)$.

Let us prove that $F$ is Lipschitz. Given $u,\,v\in B$, we have
$$
\ba{lcl}
\ds{\|F(u)-F(v)\|_{H^s(\R)}}&\leq&\ds{\|\p_x\Lambda^{-2}(u^2-v^2)\|_{H^s(\R)}+\f{1}{2}\|\p_x\Lambda^{-2}(u_x^2-v_x^2)\|_{H^s(\R)}}\\
\\&+&\ds{\|\p_x\Lambda^{-2}(h(u)-h(v))\|_{H^s(\R)}+|\lambda|\|u-v\|_{H^s(\R)}}.
\ea
$$
Applying lemmas \ref{lema2.3} and \ref{lema2.4} to the previous inequality, we obtain
\bb\label{2.4.1}
\ba{lcl}
\ds{\|F(u)-F(v)\|_{H^s(\R)}}&\leq&\ds{c_{s,R}\|(u+v)(u-v)\|_{H^{s-1}(\R)}+c_{s,R}\|(u+v)_x(u-v)_x\|_{H^{s-1}(\R)}}\\
\\&&\ds{+\|h(u)-h(v)\|_{H^{s-1}(\R)}+|\lambda|\|u-v\|_{H^s(\R)}.}
\ea
\ee
Since $u,v\in B$, we conclude that $h$ is locally Lipschitz and, as a consequence, we have $$\|h(u)-h(v)\|_{H^{s-1}(\R)}\leq c_{s,R} \|h(u)-h(v)\|_{H^{s}(\R)}\leq c_{s,R}\|u-v\|_{H^s(\R)}.$$

Applying Lemma \ref{lema2.4} and the fact that $\p_x\in{\cal L}(H^s(\R),H^{s-1}(\R))$, we can estimate 
$$
\ba{lcl}\|(u+v)_x(u-v)_x\|_{H^{s-1}(\R)}&\leq& c_{s,R}\|(u+v)_x\|_{H^{s-1}(\R)}\|(u-v)_x\|_{H^{s-1}(\R)}\\
\\
&\leq& c_{s,R}(\|u\|_{H^s(\R)}+\|v\|_{H^s(\R)})\|(u-v)\|_{H^{s}(\R)}.
\ea
$$

From Lemma \ref{lema2.5} we have $$\|(u+v)(u-v)\|_{H^{s-1}(\R)}\leq c_{s,R}(\|u\|_{H^s(\R)}+\|v\|_{H^s(\R)})\|(u-v)\|_{H^{s-1}(\R)}\leq c_{s,R}\|(u-v)\|_{H^{s}(\R)}.$$

Substituting the estimates above into the inequality \eqref{2.4.1} we prove the desired result. 
\end{proof}

\begin{corollary}
If $F$ is the function given by \eqref{2.3.3}, then $\|F(u)\|_{H^s(\R)}\leq c$, for some $c>0$.
\end{corollary}

\begin{proof}
If follows from the fact that $\|u\|_{H^s(\R)}\leq R$ and $F(0)=0$.
\end{proof}

\begin{proposition}\label{prop2.2}
Let $F:B\rightarrow H^{s-1}(\R)$ be the function given by \eqref{2.3.3}. Then $F$ is Lipschitz. 
\end{proposition}
\begin{proof} We begin with observing that
$$
\ba{lcl}
\ds{\|F(u)-F(v)\|_{H^{s-1}(\R)}}&\leq&\ds{\|(u+v)(u+v)\|_{H^{s-2}(\R)}+\f{1}{2}\|(u+v)_x(u-v)_x\|_{H^{s-2}(\R)}}\\
\\&+&\ds{\|h(u)-h(v)\|_{H^{s-2}(\R)}+|\lambda|\|u-v\|_{H^{s-1}(\R)}}
\ea
$$
The result is then proved in a similar way as done in Proposition \ref{prop2.1} and for this reason we ommit it.
\end{proof}

{\bf Proof of Theorem \ref{teo1.1}}
Firstly note that the equation in \eqref{1.0.2} is equivalent to \eqref{2.3.2}. By propositions \ref{prop2.1} and \ref{prop2.2} we see that the function $F$ given by \eqref{2.3.3} is bounded in bounded sets of $H^{s}(\R)$ and Lipschitz in $H^{s-1}(\R)$. Moreover, once $H^{s}(\R)$ and $H^{s-1}(\R)$ are Hilbert spaces, then they are reflexive and $H^s(\R)$ is continuously and densely embedded in $H^{s-1}(\R)$. Additionally, the operator $u\p_x$ generates a one-parameter $C_0$ semi-group in $H^{s}(\R)$ ({\it e.g}, see \cite{liu2011,mustafa,blanco} for the proof of this affirmation). These observations show that \eqref{2.3.2} subjected to the initial condition $u(0,x)=u_0(x)$ satisfies the conditions required in \cite{kato} (see also conditions {\bf C1}--{\bf C5} in \cite{igor-jde} for a more to the point presentation). 
The result is then a consequence of Kato's theorem, see Theorem 6 in \cite{kato}.\quad\quad$\square$

\section{Invariants}\label{sec3}

In this section we study the structural properties of the equation \eqref{2.3.1}, for any value of $\lambda$. We begin with the Lie point symmetries, which tell us that the most general invariant solutions admitted by the equation are travelling waves. Next, we look for the construction of conserved currents of the equation, which provides us invariants with respect to $t$. 

\subsection{Lie symmetries}

In what follows we present a {\it to the point} overview of symmetries and conservation laws for equations with two independent variables $(t,x)$ and one dependent one $u$. For further details, see \cite{bk,ibrabook,olverbook}.

A smooth function depending on $(t,x,u)$ and derivatives of $u$ up to a finite, but arbitrary, order is called differential function, while the set of all of differential functions is denoted by ${\cal A}$. The order of a differential function $F\in{\cal A}$ is the order of the highest derivative appearing in it.

There are some natural operators acting on ${\cal A}$: the total derivative operators with respect to $t$ and $x$ are given by
$$
\ba{lcl}
D_{t}&=&\ds{\f{\p}{\p t}+u_{t}\f{\p}{\p u}+u_{tt}\f{\p}{\p u_{t}}+u_{tx}\f{\p}{\p u_{x}}+\cdots},\\
\\
D_x&=&\ds{\f{\p}{\p x}+u_{x}\f{\p}{\p u}+u_{xt}\f{\p}{\p u_{t}}+u_{xx}\f{\p}{\p u_{x}}+\cdots},
\ea
$$
and the Euler-Lagrange operator
\bb\label{3.1.1}
{\cal E}_u:=\f{\p}{\p u}-D_x\f{\p}{\p u_x}-D_t\f{\p}{\p u_t}+D_x^2\f{\p}{\p u_{xx}}+\cdots .
\ee
Let $(t,x,u)\mapsto (\overline{t}(t,x,u,\epsilon),\overline{x}(t,x,u,\epsilon),\overline{u}(t,x,u,\epsilon))$ be a one-parameter group of transformations that at $\epsilon=0$ corresponds to the identity. Assuming that such transformation is analytic with respect to the parameter $\epsilon$, we have
\bb\label{3.1.3}
\ba{lcl}
\overline{t}(t,x,u,\epsilon)&=&t+\epsilon\tau(t,x,u)+\mathcal{O}(\epsilon^2),\\
\\
\overline{x}(t,x,u,\epsilon)&=&x+\epsilon\xi(t,x,u)+\mathcal{O}(\epsilon^2),\\
\\
\overline{u}(t,x,u,\epsilon)&=&u+\epsilon\eta(t,x,u)+\mathcal{O}(\epsilon^2).
\ea
\ee
The coefficients $\tau,\,\xi,\,\eta$, which depend only on $(t,x,u)$, define the infinitesimal generator
\bb\label{3.1.4}
X=\tau\f{\p}{\p t}+\xi\f{\p}{\p x}+\eta\f{\p}{\p u}
\ee
of the group of transformations. The transformations \eqref{3.1.3} are Lie point symmetries of the equation in \eqref{2.3.1} if and only if
\bb\label{3.1.5}
X^{(3)}(u_t-u_{txx}+\lambda(u-u_{xx})-uu_{xxx}-2u_xu_{xx}+3u^2u_x)=0,
\ee
whenever \eqref{2.3.1} holds or, what is the same, $$
\ba{l}
X^{(3)}(u_t-u_{txx}+\lambda(u-u_{xx})-uu_{xxx}-2u_xu_{xx}+3u^2u_x)\\
\\=\mu(u_t-u_{txx}+\lambda(u-u_{xx})-uu_{xxx}-2u_xu_{xx}+3u^2u_x),
\ea$$
for some $\mu\in{\cal A}$. The operator $X^{(3)}$ above is the third order prolongation of the generator \eqref{3.1.4} and it is given by (here we only write the needed components of the operator. In the general case, we have more components in the prolongation.)
$$
X^{(3)}=X+\zeta^t\f{\p}{\p u_t}+\zeta^x\f{\p}{\p u_x}+\zeta^{xx}\f{\p}{\p u_{xx}}+\zeta^{xxx}\f{\p}{\p u_{xxx}}+\zeta^{txx}\f{\p}{\p u_{txx}},
$$
where
$$
\ba{lcl}
\zeta^{t}&=&D_{t}(\eta)-(D_{t}\tau) u_{t}-(D_{t}\xi) u_{x},\\
\\
\zeta^{x}&=&D_{x}(\eta)-(D_{x}\tau) u_{t}-(D_{x}\xi) u_{x},\\
\\
\zeta^{xx}&=&D_{x}(\zeta^{x})-(D_{x}\tau) u_{xt}-(D_{x}\xi) u_{xx},\\
\\
\zeta^{xxx}&=&D_{x}(\zeta^{xx})-(D_{x}\tau) u_{xxt}-(D_{x}\xi) u_{xxx},\\
\\
\zeta^{txx}&=&D_{t}(\zeta^{xx})-(D_{t}\tau) u_{ttx}-(D_{t}\xi) u_{txx}.
\ea
$$

The condition \eqref{3.1.5} is called invariant condition. From it we prove the following result:
\begin{theorem}\label{teo3.1}
The Lie point symmetries of the equation \eqref{2.3.1} are generated by the operators 
\bb\label{3.1.8}
X_1=\f{\p}{\p t},\quad X_2=\f{\p}{\p x},
\ee
which correspond to translations in $t$ and $x$, respectively.
\end{theorem}

The fluxes determined by the operators in \eqref{3.1.8} are, respectively, given by $e^{\epsilon X_1}(t,x,u)=(t+\epsilon,x,u)$ and $e^{\epsilon X_2}(t,x,u)=(t,x+\epsilon,u)$, meaning that the only symmetries of \eqref{2.3.1} are translations. Taking the linear combination $X_1+cX_2$, we obtain $e^{\epsilon (X_1+cX_2)}(t,x,u)=(t+\epsilon,x+c\epsilon,u)$, which implies that if $u$ is a solution of \eqref{2.3.1} invariant under this flux, then $u=f(x-ct)$, for some real function $f$.
\subsection{Conserved currents}

We recall that a conserved current for a partial differential equation $F=0$ with two independent variables $(t,x)$ and a dependent variable $u$ is a pair $C=(C^0,C^1)$, where $C^0,C^1\in{\cal A}$, such that $\mathop{\rm Div}(C):=D_t C^0+D_x C^1$ vanishes identically on the solutions of the equation. It is possible to show that (see \cite{olverbook}, page 266) it is equivalent to
\bb\label{3.2.1}
D_{t}C^0+D_{x}C^1=\phi (F),
\ee
where $\phi\in{\cal A}$ is called characteristic of the conservation law, while the expression in \eqref{3.2.1} is referred as the characteristic form of the conservation law corresponding to the conserved current $C$. The order of the conserved current $C$ is defined as the maximum of the orders of its components.

Since $\mathop{\rm Div}(C)\in\ker{({\cal E}_u)}$, for any $C=(C^0,C^1)$ (see \cite{olverbook}, Theorem 4.7, page 248), we can determine the function $\phi$ in \eqref{3.2.1} through the equation
\bb\label{3.2.2}
{\cal E}_u(\phi F)=0,
\ee
where ${\cal E}_u$ is the Euler-Lagrange operator \eqref{3.1.1}.

Let us assume that $\phi=\phi(t,x,u,u_t,u_x,u_{tt},u_{tx},u_{xx})$. We will use the condition \eqref{3.2.2} to obtain conserved currents for equation \eqref{2.3.1}. We begin with the case $\lambda=0$. Substituting 
\bb\label{3.2.3}
F:=u_{t}-u_{txx}-uu_{xxx}-2u_{x}u_{xx}+3u^2u_{x}=0
\ee
and $\phi$ into \eqref{3.2.2}, we obtain the set of equations 
\bb\label{3.2.4}
\left\{\ba{l}
\phi_{uu}+6u\phi_{u_{tx}}=0,\\
\phi_{u_{x}}-\phi_{u_{tx}}u_{x}=0,\\
\phi_{u_{xx}}-u\phi_{u_{tx}}=0,\\
\phi_{u_{tx}u_{tx}}=0,\\
\phi_{uu_{tx}}=0,\\
\phi_{u_{tt}}=0,\\
\phi_{u_{t}}=0.
\end{array}\right.
\ee
The solution of system \eqref{3.2.4}  is
\bb\label{3.2.5}
\phi=c_1+c_2u +c_3\left(u^{3} - \dfrac{1}{2}u_{x}^{2} - u_{tx} - uu_{xx}\right),
\ee
where $c_1,c_2$ and $c_3$ are arbitrary constants. Multiplying \eqref{3.2.5} and $F$ given by \eqref{3.2.3} we obtain, after some calculations,
$$
\ba{lcl}
\phi F&=&c_1\left[D_t\left( u-u_{xx}\right)+D_x\left( u^{3} - \dfrac{1}{2}u_{x}^{2} - uu_{xx} \right)\right]\\
\\

&+&c_2\left[D_t\left( \dfrac{u^{2}+u_{x}^{2}}{2}\right)+D_x\left( \dfrac{3}{4}u^{4} - uu_{tx} -u^{2}u_{xx}\right)\right]\\
\\
&+&c_3\left[D_t\left(  \dfrac{1}{4}u^{4} +\dfrac{1}{2}uu_{x}^{2}\right)+D_x\left( \dfrac{1}{2}u^{6} + \dfrac{1}{8}u_{x}^{4} - \dfrac{1}{2}u_{t}^{2} - uu_{t}u_{x} - \dfrac{1}{2}u^{3}u_{x}^{2} \right.\right.\\
\\
&&\left.\left.+ \dfrac{1}{2}uu_{x}^{2}u_{xx}+ \dfrac{1}{2}u_{x}^{2}u_{tx} + \dfrac{1}{2}u_{tx}^{2}+ \dfrac{1}{2}u^{2}u_{xx}^{2} - u^{4}u_{xx} +uu_{xx}u_{tx} - u^{3}u_{tx} \right)\right].
\ea
$$

We have proved the following result:
\begin{theorem}\label{teo3.2}
The conserved currents up to second order of the equation \eqref{1.0.1} is generated by linear combinations of the currents
\bb\label{3.2.7}
\left( u-u_{xx}, u^{3} - \dfrac{1}{2}u_{x}^{2} - uu_{xx} \right),
\ee
$$
\left(\dfrac{u^{2}+u_{x}^{2}}{2},  \dfrac{3}{4}u^{4} - uu_{tx} -u^{2}u_{xx} \right),
$$
and
$$
\ba{lcl}
\left(\dfrac{1}{4}u^{4} +\dfrac{1}{2}uu_{x}^{2}\right.&,&  \dfrac{1}{2}u^{6} + \dfrac{1}{8}u_{x}^{4} - \dfrac{1}{2}u_{t}^{2} - uu_{t}u_{x} - \dfrac{1}{2}u^{3}u_{x}^{2} + \dfrac{1}{2}uu_{x}^{2}u_{xx}\\
\\
&+&\left.\dfrac{1}{2}u_{x}^{2}u_{tx} + \dfrac{1}{2}u_{tx}^{2}+ \dfrac{1}{2}u^{2}u_{xx}^{2} - u^{4}u_{xx} +uu_{xx}u_{tx} - u^{3}u_{tx} \right).
\ea
$$
\end{theorem}
We observe that if we make the change $u(t,x)=\phi(x-ct)$, $c\neq0$, the divergence of the current \eqref{3.2.7} gives
\bb\label{3.2.10}
\f{d}{dz}[c(\phi-\phi'')-\phi^3+\f{1}{2}(\phi')^2+\phi\phi'']=0,
\ee
where the prime $'$ in \eqref{3.2.10} means derivative with respect to $z$. Equation \eqref{3.2.10} is nothing but a first integral of the ODE obtained from \eqref{1.0.1} assuming that $u(t,x)=\phi(x-ct)$, that is the most general solution of \eqref{1.0.1} invariant under translations in $t$ and $x$, which is generated by the generators given in Theorem \ref{teo3.1}.

A periodic solution of the equation \eqref{3.2.10} is (see equation (9) in \cite{kelly} or formula (1.8) in \cite{natali}) 
\bb\label{3.2.11}
\phi(z)=\al+\be\sn^2{\left(\f{2 K(k)z}{L};k\right)},
\ee
where $\al$ and $\be$ are certain functions of $L>0$ (for the meaning of $L$, see \cite{kelly,natali}), $K(k)$ represents the complete elliptic integral of first kind and $\sn$ denotes the Jacobi snoidal function (for further properties of these functions, see the Appendix A in \cite{chen2020} or \cite{kelly}). 

Recalling the relation $k^2\sn^2+\dn^2=1$, where $\dn$ is the Jacobi dnoidal function, we can express \eqref{3.2.11} (see equation (1.9) in \cite{natali}) as
$$
\phi(z)=a+b\left[\dn^2{\left(\f{2 K(k)z}{L};k\right)}-\f{E(k)}{K(k)}\right],
$$
where $E(\cdot)$ is the complete elliptic integral of the second kind and this representation of the solutions satisfies the relation
$$
\f{1}{L}\int_{0}^L\phi(z)dz=a.
$$

The solution $u(t,x)=\phi(x-ct)$, where $\phi$ is given by \eqref{3.2.11}, is orbitally stable in the energy space $H^1_{\text{per}}([0,L])$, see Theorem 2.8 in \cite{natali}. Orbital instability of periodic waves of equation \eqref{1.0.1} was investigated in \cite{kelly}, see also \cite{natali}.

\begin{corollary}\label{cor3.1}
Let $u$ be a sufficiently smooth solution of \eqref{1.0.1} such that $u(t,x)$ and its derivatives up to second order go to $0$ as $x\rightarrow\pm\infty$, and $u_0:=u(0,x)$. Then the functionals 
\bb\label{3.2.12}
{\cal H}_0[u]=\int_\R u dx,
\ee
\bb\label{3.2.13}
{\cal H}[u]=\f{1}{2}\int_\R (u^2+u_x^2) dx
\ee
and
\bb\label{3.2.14}
{\cal H}_1[u]=\int_\R \left(\dfrac{1}{4}u^{4} +\dfrac{1}{2}uu_{x}^{2}\right) dx
\ee
are independent of $t$. In particular, ${\cal H}_0[u]={\cal H}_0[u_0]$ and ${\cal H}[u]={\cal H}[u_0]$.
\end{corollary}
\begin{proof}
Taking the divergence of the current \eqref{3.2.7} we have
$$
D_t\left( u-u_{xx}\right)+D_x\left( u^{3} - \dfrac{1}{2}u_{x}^{2} - uu_{xx} \right)=0.
$$
Integrating the divergence above with respect to $x$, we obtain
$$
\int_\R D_t\left( u-u_{xx}\right)dx=-\left.\left(u^{3} - \dfrac{1}{2}u_{x}^{2} - uu_{xx}\right)\right|_{x=-\infty}^{x=+\infty}=0.
$$
Noting that
$$
\int_\R D_t\left( u-u_{xx}\right)dx=\f{d}{dt}\int_\R\left( u-u_{xx}\right)dx=\f{d}{dt}\left({\cal H}_0[u]-\left.u_x\right|_{x=-\infty}^{x=+\infty}\right)=\f{d}{dt}{\cal H}_0[u],
$$
we are forced to conclude that
$$\f{d}{dt}{\cal H}_0[u]=0.$$
In particular, this implies that ${\cal H}_0[u]={\cal H}_0[u_0]$. The other conserved quantities are obtained in a similar form and, therefore, the remaining proofs are omitted.
\end{proof}

The quantities \eqref{3.2.12}--\eqref{3.2.14} are known as {\it conserved quantities}. These conserved quantities are relevant for several reasons. To name a few, they play vital role in the investigation of orbital stability/instability of solutions of the equation, see \cite{kelly,natali}. In our case, they are relevant to establish the wave breaking phenomenon, see Theorem \ref{teo1.3}.

In our paper we are concerned with the Cauchy problem \eqref{1.0.2} and for this reason the functionals are considered over the real line $\R$. We note, however, that if we were interested in periodic problems, we should replace $\R$ by $[0,L]$ in \eqref{3.2.12}--\eqref{3.2.14} and assume that $u$ and its derivatives vanish at $x=0$ and $x=L$. In particular, we can then construct the augmented Lyapunov functional $-{\cal H}_1[u]+c{\cal H}[u]-A{\cal H}_0[u]$, where A is the constant obtained after integrating \eqref{3.2.10}, which is relevant in the study of orbital stability/instability of (periodic) solutions of the solutions of \eqref{1.0.1}, see \cite{kelly,natali}.

It is worth mentioning that the invariants \eqref{3.2.12}--\eqref{3.2.14} are also conserved quantities for the equation \eqref{1.0.1} even for other domains such that the solutions vanish on the corresponding boundary. The demonstration is just the same for the Theorem \ref{teo3.2} replacing $\R$ by the respective domain.

The proof of Corollary \ref{cor3.1} gives another demonstration for Lemma 1 in \cite{kelly} without using {\it ad hoc} procedures but, instead, using structural information of the equation. Actually, while most papers dedicated to qualitative properties of solutions obtain invariants using the same procedure employed in \cite{kelly}, in our case we construct the invariants \eqref{3.2.12}--\eqref{3.2.14} from the structural information coming from the conserved currents associated to the equation \eqref{1.0.1}.

\begin{theorem}\label{teo3.3}
Up to second order, the conserved currents of the equation \eqref{2.3.1} are
$$
\left( e^{\lambda t}\left(u-u_{xx}\right),e^{\lambda t}\left( u^{3} - \dfrac{1}{2}u_{x}^{2} - uu_{xx}\right) \right)
$$
and
$$
\left(e^{2\lambda t}\left(\dfrac{u^{2}+u_{x}^{2}}{2}\right),  e^{2\lambda t}\left(\dfrac{3}{4}u^{4} - uu_{tx} -u^{2}u_{xx}-\lambda uu_x\right) \right).
$$
\end{theorem}
\begin{proof}
A straightforward calculation shows that equation \eqref{2.3.1} can be rewritten in the form
$$
D_t\left( u-u_{xx}\right)+D_x\left( u^{3} - \dfrac{1}{2}u_{x}^{2} - uu_{xx} \right)+\lambda(u-u_{xx})=0,
$$
which can be rewritten as
$$
D_t\left( e^{\lambda t}\left(u-u_{xx}\right)\right)+D_x\left(e^{\lambda t}\left( u^{3} - \dfrac{1}{2}u_{x}^{2} - uu_{xx}\right) \right)=0.
$$
The second relation comes from the identity
$$
\ba{lcl}
0&=&D_t\left( \dfrac{u^{2}+u_{x}^{2}}{2}\right)+D_x\left( \dfrac{3}{4}u^{4} - uu_{tx} -u^{2}u_{xx}\right)+\lambda(u^2-uu_{xx})\\
\\
&=&D_t\left( \dfrac{u^{2}+u_{x}^{2}}{2}\right)+D_x\left( \dfrac{3}{4}u^{4} - uu_{tx} -u^{2}u_{xx}-\lambda uu_x\right)+\lambda(u^2+u_x^2)
\ea
$$
and the same steps used before.
\end{proof}
\begin{corollary}\label{cor3.2}
Let $u$ be a sufficiently smooth solution of \eqref{2.3.1} such that $u(t,x)$ and its derivatives up to second order go to $0$ as $x\rightarrow\pm\infty$, and $u_0:=u(0,x)$. Then the functionals 
$$
{\cal H}_0[u]=e^{-\lambda t}\int_\R u dx,
$$
and
\bb\label{3.2.18}
{\cal H}[u]=\f{e^{-2\lambda t}}{2}\int_\R \left(u^2+u_x^2\right) dx.
\ee

are independent of $t$. In particular, ${\cal H}_0[u]\leq{\cal H}_0[u_0]$ and ${\cal H}[u]\leq{\cal H}[u_0]$.
\end{corollary}
\begin{proof}
The proof is similar as that done in Corollary \eqref{cor3.1} and, therefore, is omitted.
\end{proof}

It is worth mentioning that if $\lambda>0$, the equation \eqref{3.2.18} implies
$$
{\cal H}[u]\leq \f{\|u_0\|_{H^1(\R)}^2}{2},
$$
which will be of great relevance in our work, mainly in the investigation of wave breaking.

Combining the last observation with the Sobolev Embedding Theorem, we prove the following corollary.
\begin{corollary}\label{cor3.3}
If $u_0\in H^s(\R)$, $s>3/2$ and $\lambda>0$, then the $H^1(\R)-$norm of the corresponding solution to the problem \eqref{1.0.2} is bounded from above by $\|u_0\|_{H^1(\R)}$.
\end{corollary}

\section{Wave breaking}\label{sec4}

Corollary \ref{cor3.3} implies that the shape of the solutions of \eqref{1.0.2} remains bounded provided that $\lambda>0$ and $u_0\in H^s(\R)$, with $s>3/2$, that is, condition \eqref{1.1.2} is satisfied. Therefore, if the solutions of \eqref{1.0.2} blows-up, then we should look for singularities in $u_x$ like \eqref{1.1.1}, which is a manifestation of wave breaking.

We observe that Theorem \ref{teo1.2} gives us conditions to avoid the appearance of such singularity. In order to prove it, we rewrite \eqref{2.3.1} in the more convenient form
\bb\label{4.0.1}
m_t+um_x+2u_xm+\lambda m+\p_x h(u)=0,
\ee
where $h(u)$ is the function given in \eqref{2.3.4} and $m=u-u_{xx}$.

{\bf Proof of Theorem \ref{teo1.2}.} Note that 
\bb\label{4.0.2}
\f{d}{d t}\|m\|^2_{H^1(\R)}=\f{d}{d t}\|m\|^2_{L^2(\R)}+\f{d}{d t}\|m_x\|^2_{L^2(\R)}=2\left(\langle m,m_t\rangle_{L^2(\R)}+\langle m_x,m_{tx}\rangle_{L^2(\R)}\right).
\ee
Substituting \eqref{4.0.1} into \eqref{4.0.2}, we have
$$
\ba{lcl}
\langle m,m_t\rangle_{L^2(\R)}&=&\ds{-\langle m,um_x\rangle_{L^2(\R)}-2\langle u_x,m^2\rangle_{L^2(\R)}}\\
\\
&&\ds{-\lambda\langle m,m\rangle_{L^2(\R)}+\langle m,\p_x h(u)\rangle_{L^2(\R)}}\\
\\
&=&\ds{-\f{3}{2}\langle u_x,m^2\rangle-\lambda\|m\|^2_{L^2(\R)}+\langle m,\p_x h(u)\rangle},
\ea
$$
where we used the relations $\langle m,u_xm\rangle=\langle u_x,m^2\rangle$ and $\langle m,um_x\rangle=\langle u,m m_x\rangle=-\langle u_x,m^2\rangle/2$

Differentiating \eqref{4.0.1} with respect to $x$ and substituting the result into $\langle m_x,m_{tx}\rangle_{L^2(\R)}$ and proceeding in a similar way as before, we obtain
$$
\ba{lcl}
\langle m_x,m_{tx}\rangle_{L^2(\R)}&=&\ds{-\langle m_x,um_{xx}\rangle_{L^2(\R)},m_{xx}\rangle_{L^2(\R)}-\langle m_x,u_xm_x\rangle_{L^2(\R)}}\\
\\
&=&\ds{-\f{5}{2}\langle u_x,m_x^2\rangle_{L^2(\R)}+\langle u_{x},m^2\rangle_{L^2(\R)}}\\
\\
&&\ds{-\lambda\|m_x\|^2_{L^2(\R)}+\langle m_x,\p_{x}^2}h(u)\rangle_{L^2(\R)},
\ea
$$
where we used $\langle u_{xx},mm_x\rangle_{L^2(\R)}=\langle u,\p_{x} m^2\rangle_{L^2(\R)}/2-\langle m,mm_x\rangle_{L^2(\R)}=-\langle u_x,m^2\rangle_{L^2(\R)}/2$.

Therefore, we have
\bb\label{4.0.5}
\ba{lcl}
\ds{\f{d}{d t}\|m\|^2_{H^1(\R)}}&=&\ds{-\langle u_x,m^2\rangle_{L^2(\R)}-5\langle u_x,m_x^2\rangle_{L^2(\R)}-2\lambda \|m\|_{H^1(\R)}^2}\\
\\
&&\ds{+2\left(\langle m,\p_x^2 h(u)\rangle_{L^2(\R)}+\langle m_x,\p_x^2 h(u)\rangle_{L^2(\R)}\right)}.
\ea
\ee

Defining $I:=\langle m,\p_x^2 h(u)\rangle_{L^2(\R)}+\langle m_x,\p_x^2 h(u)\rangle_{L^2(\R)}$, a straightforward calculation reads
$$
I\leq \f{\|\Lambda^2\p_x h(u)\|^2_{L^2(\R)}+\|m\|_{L^2(\R)}^2}{2}.
$$
Finally, we also have 
$$\|\Lambda^2\p_x h(u)\|_{L^2(\R)}\leq\|\p_x h(u)\|_{H^2(\R)}\leq\|h(u)\|_{H^3(\R)}.$$ Invoking Lemma \ref{lema2.3}, we conclude that $\|h(u)\|_{H^3(\R)}\leq c_1\|u\|_{H^3(\R)}$, for a certain constant $c_1$, and then 
$$\|h(u)\|_{H^3(\R)}\leq c_1\|u\|_{H^3(\R)}=c_1\|\Lambda^{-2}u\|_{H^1(\R)}=c_1\|m\|_{H^1(\R)}.$$
Since $\|m\|_{L^2(\R)}\leq \|m\|_{H^1(\R)}$ we can infer that
\bb\label{4.0.6}
\langle m,\p_x^2 h(u)\rangle_{L^2(\R)}+\langle m_x,\p_x^2 h(u)\rangle_{L^2(\R)}\leq c\|m\|^2_{H^1(\R)},
\ee
for some positive constant $c$. In addition, we have
\bb\label{4.0.7}
\ba{lcl}
\ds{-\langle u_x,m^2\rangle_{L^2(\R)}-5\langle u_x,m_x^2\rangle_{L^2(\R)}}&=&\ds{\int_{\R}(-u_x)(m^2+5m_x^2)dx\leq k\int_{\R}(m^2+5m_x^2)dx}\\
\\
&\leq& 5k\|m\|_{H^1(\R)}^2.
\ea
\ee
Substitution of \eqref{4.0.6} and \eqref{4.0.7} into \eqref{4.0.5} reads
$$
\f{d}{d t}\|m\|^2_{H^1(\R)}\leq (5k+c+2\lambda)\|m\|_{H^1(\R)}^2,
$$
which, after integration, yields the inequality
$$
\|m\|_{H^1(\R)}^2\leq e^{(5k+c+2\lambda) t} \|m_0\|^2_{H^1(\R)},
$$
and this concludes the demonstration.
\quad\quad\quad\quad\quad\quad\quad\quad\quad\quad\quad\quad\quad\quad\quad\quad\quad\quad\quad\quad\quad\quad\quad\quad\quad\quad\quad\quad\,\,$\square$

Recalling that $H^3(\R)\subseteq H^s(\R)$, for any $s\leq 3$, and the fact that this embedding is dense and continuous, we prove the following consequence of Theorem \ref{teo1.2}.

\begin{corollary}\label{cor4.1}
Let $u_0\in H^s$, $s\geq3/2$, and $u$ be the corresponding solution to \eqref{1.0.2} with initial data $u(0,x)=u_0(x)$. Assume that $u_x>-k$, for some positive constant $k$. Then $\|u\|_{H^s(\R)}\leq e^{\kappa t}\|u_0\|_{H^s(\R)}$, for a certain positive constant $\kappa$.
\end{corollary}

From Theorem \ref{teo1.2} and Corollary \ref{cor4.1} we conclude that for the emergence of wave breaking we must look for solutions having no lower bound in their $x-$derivatives. We then begin with the following result:
\begin{lemma}\label{lema4.1}
Let $T>0$ and $v\in C^1([0,T),H^2(\R))$ be a given function. Then, for any $t\in[0,T)$, there exists at least one point $\xi(t)\in\R$ such that
$$
y(t)=\inf_{x\in\R}{v_x(t,x)}=v_x(t,\xi(t))
$$
and the function $y$ is almost everywhere differentiable (a.e) in $(0,T)$, with $y'(t)=v_{tx}(t,\xi(t))$ {\it a.e.} on $(0,T)$.
\end{lemma}
\begin{proof}
See Theorem 2.1 in \cite{const1998-2} or Theorem 5 in \cite{escher}.
\end{proof}

We recall that equation \eqref{4.0.1} is equivalent to
\bb\label{4.0.9}
u_t+uu_x+\partial_x\Lambda^{-2}\left(u^2+\frac{u_x^2}{2}+h(u)\right)+\lambda u=0,
\ee
see Proposition \ref{prop2.1}. Taking the $x$ derivative of the equation \eqref{4.0.9} and using the relation $\partial_x^2\Lambda^{-2}=\Lambda^{-2}-1$ we obtain 
\bb\label{4.0.10}
u_{tx}+\frac{u_x^2}{2}+\lambda u_x=u^2-uu_{xx}+h(u)-\Lambda^{-2}\left(u^2+\frac{u_x^2}{2}\right)-\Lambda^{-2}(h(u)).
\ee
Our aim is to evaluate \eqref{4.0.10} at $(t,\xi(t))$, where $\xi(t)$ is the point assured by Lemma \ref{lema4.1}. Firstly, we note that (see \cite{const1998-2}, pages 239--240; or \cite{escher}, pages 106--107)
$$\Lambda^{-2}\left(u^2+\frac{u_x^2}{2}\right)\geq \frac{u^2}{2}.$$

Secondly, if we assume that at $t=0$, then $u(0,x)=u_0(x)$, we conclude that
$$2u^2(t,\xi(t))=2\left(\int_{-\infty}^{\xi(t)}(uu_x)(t,y)dy-\int_{\xi(t)}^{\infty}(uu_x)(t,y)dy\right)$$
and then
$$
\begin{array}{rcl}
\ds{2u^2(t,\xi(t))}
&\leq&\ds{\left|2\left(\int_{-\infty}^{\xi(t)}(uu_x)(t,y)dy-\int_{\xi(t)}^{\infty}(uu_x)(t,y)dy\right)\right|\leq\left|\int_{-\infty}^{\xi(t)}2(uu_x)(t,y)dy\right|}\\
\\&+&\ds{\left|\int_{\xi(t)}^{\infty}2(uu_x)(t,y)dy\right|\leq\left|\int_{-\infty}^{\xi(t)}\left(u^2+u_x^2\right)(t,y)dy\right|+\left|\int_{\xi(t)}^{\infty}\left(u^2+u_x^2\right)(t,y)dy\right|}\\
\\
&\leq&\ds{\int_{-\infty}^{\xi(t)}\left(u^2+u_x^2\right)(t,y)dy+\int_{\xi(t)}^{\infty}\left(u^2+u_x^2\right)(t,y)dy}\\
\\
&\leq&\ds{\int_{\mathbb{R}}\left(u^2+u_x^2\right)(t,y)dy=e^{-2\lambda t}\|u_0\|_{H^1(\R)}^2}.
\end{array}
$$
Therefore, we are forced to conclude that 
$$|u(t,\xi(t))| \leq \frac{\sqrt{2}}{2}e^{-\lambda t}\|u_0 \|_{H^1(\R)}$$
and we then have the following lower and upper bounds to $u(t,\xi(t))$:
\bb\label{4.0.11}
-\frac{\sqrt{2}}{2}e^{-\lambda t}\|u_0 \|_{H^1(\R)}\leq u(t,\xi(t))\leq \frac{\sqrt{2}}{2}e^{-\lambda t}\|u_0 \|_{H^1(\R)}.
\ee
Let $h$ be the function given in \eqref{2.3.4}. The estimate \eqref{4.0.11} implies
$$
\ba{lcl}
\ds{-\frac{1}{2}e^{-2\lambda t}\|u_0 \|_{H^1(\R)}^2\left(\frac{\sqrt{2}}{2}e^{-\lambda t}\|u_0 \|_{H^1(\R)}+\frac{3}{2}\right)}&\leq& h(u(t,\xi(t)))\\
\\&\leq& \ds{\frac{1}{2}e^{-2\lambda t}\|u_0 \|_{H^1(\R)}^2\left(\frac{\sqrt{2}}{2}e^{-\lambda t}\|u_0 \|_{H^1(\R)}-\frac{3}{2}\right)}
\ea
$$
and
$$
-\Lambda^{-2}(h(u(t,\xi(t))))\leq\left(\frac{\sqrt{2}}{4}e^{-3\lambda t}\|\varphi \|_{H^1(\R)}^3+\frac{3}{4}e^{-2\lambda t}\|\varphi \|_{H^1(\R)}^2\right).
$$

Finally, substituting the latter inequality into \eqref{4.0.10} evaluated at $(t,\xi(t))$ and taking into account that $u_{xx}(t,\xi(t))\equiv0$, we obtain the inequality
\bb\label{4.0.12}
\ds{u_{tx}(t,\xi(t))+\frac{u_x^2}{2}(t,\xi(t))+\lambda u_x(t,\xi(t))\leq \frac{\sqrt{2}}{2}e^{-3\lambda t}\|u_0 \|_{H^1(\R)}^3+\frac{1}{4}e^{-2\lambda t}\|u_0\|_{H^1(\R)}^2.}
\ee

Let us now invoke Lemma \ref{lema4.1} and assume that $\lambda\geq0$. Defining $$y(t)=\inf_{x\in\mathbb{R}}(u_x(t,x))=u_x(t,\xi(t)),$$ relation \eqref{4.0.12} yields
$$
\ds{y'(t)+\frac{y(t)^2}{2}+\lambda y(t) \leq \frac{\sqrt{2}}{2}\|u_0 \|_{H^1(\R)}^3+\frac{1}{4}\|u_0\|_{H^1(\R)}^2.}
$$

{\bf Proof of Theorem \ref{teo1.3}}: Let 
$$
\psi(u_0):=\frac{\sqrt{2}}{2}\|u_0 \|_{H^1(\R)}^3+\frac{1}{4}\|u_0\|_{H^1(\R)}^2
$$
and
$$
\epsilon_0=1-\f{2\sqrt{2}\|u_0\|_{H^1(\R)}^3+\|u_0\|_{H^1(\R)}^2}{2(u_0'(x_0))^2}.
$$

The last inequality can be rewritten as
\bb\label{4.0.14}
y'(t)+\lambda y(t)\leq \psi(u_0)-\f{y(t)^2}{2}.
\ee

Suppose that $u_0'(x_0)<-\sqrt{2\psi(u_0)}$, for some $x_0\in\R$. Then $2\psi(u_0)<(u_0'(x_0))^2$. We note that for each $\epsilon\in(0,\epsilon_0]$, we have the inequality
$$
\psi(u_0)\leq\f{1-\epsilon}{2}(u'_0(x_0))^2.
$$

Since $u_0'(x_0)=u_x(0,x_0)$ and $y(0)=\inf\limits_{x\in\R}u_x(0,x)$ we conclude that $y(0)\leq u'_0(x_0)$ and then
\bb\label{4.0.15}
\psi(u_0)\leq\f{1-\epsilon}{2}y(0)^2.
\ee

Substituting \eqref{4.0.15} into \eqref{4.0.14} we obtain
\bb\label{4.0.16}
y'(t)+\lambda y(t)\leq\f{1-\epsilon}{2}y(0)^2-\f{y(t)^2}{2}.
\ee

From \cite{const1998-2}, page 240, or \cite{escher}, page 108, we know that
$$
y(t)^2>\left(1-\f{\epsilon}{2}\right)y(0)^2,\quad \forall t\in\,\,[0,T),
$$
where $T$ is given by Theorem \ref{teo1.1}.

The inequality above implies 
$$y(0)^2<\f{2}{2-\epsilon}y(t)^2$$
and its substitution into \eqref{4.0.16} yields
\bb\label{4.0.17}
y'(t)+\lambda y(t)<-\f{\epsilon}{2(2-\epsilon)}y(t)^2<-\f{\epsilon}{4}y(t)^2.
\ee

Let us assume firstly that $\lambda>0$. Defining $z(t)=e^{\lambda t}y(t)$, we obtain the differential inequality
$$
\f{d}{dt}\left(\f{1}{z(t)}\right)=-\f{z'(t)}{z(t)^2}>\f{\epsilon}{4}e^{-\lambda t},
$$
which, after integration, gives
$$
\f{1}{z(t)}>\f{1}{z(0)}+\f{\epsilon}{4\lambda}(1-e^{-\lambda t}).
$$
Therefore, the function $y(t)$ satisfies
$$
\f{1}{y(t)}>\f{e^{\lambda t}}{y(0)}+\f{\epsilon}{4\lambda}(e^{\lambda t}-1)
$$
and the last inequality is equivalent to
\bb\label{4.0.18}
e^{\lambda t}\left(\f{\epsilon}{4\lambda}+\f{1}{y(0)}\right)<\f{\epsilon}{4\lambda}+\f{1}{y(t)}<\f{\epsilon}{4\lambda},
\ee
for each $\epsilon\in(0,\epsilon_0]$, due to $y(t)<0$. Assuming that $\lambda\in(0,\lambda_0)$, where $\lambda_0=-y(0)\epsilon_0/4$, we are forced to conclude that
\bb\label{4.0.19}
\f{\epsilon}{4\lambda}+\f{1}{y(0)}>\f{\epsilon}{4\lambda_0}+\f{1}{y(0)}=-\f{\epsilon-\epsilon_0}{y(0)}.
\ee

Since the inequality \eqref{4.0.19} holds for each $\epsilon\in(0,\epsilon_0]$, in particular it is true for $\epsilon=\epsilon_0$ and then, from \eqref{4.0.18} jointly with \eqref{4.0.19} and taking $\epsilon=\epsilon_0$, we obtain
$$
0<e^{\lambda t}\left(\f{\epsilon_0}{4\lambda}+\f{1}{y(0)}\right)<\f{\epsilon_0}{4\lambda},
$$
which forces $t$ to be finite. To determine the upper bound to $t$, we note that the last inequality gives
$e^{\lambda t}<(\epsilon_0y(0))/(\epsilon_0 y(0)+4\lambda)$, which implies
$$
t<\f{1}{\lambda}\ln{\left(\f{\epsilon_0 y(0)}{\epsilon_0 y(0)+4\lambda}\right)}=:T_+.
$$

Let us assume that $\lambda=0$. From \eqref{4.0.17} we obtain the inequality
$$
\f{1}{y(t)}-\f{1}{y(0)}>\f{\epsilon}{4}t
$$
and since $y(t)<y(0)<0$, we conclude that
$$
\f{\epsilon t}{4}<-\f{2}{y(0)},
$$
for each $\epsilon\in(0,\epsilon_0]$. This is enough to assure the wave breaking, while taking $\epsilon=\epsilon_0$ we have the upper bound to the maximal time of existence of the solution and conclude the demonstration of Theorem \ref{teo1.3}.

\section{Discussion and comments}

In this paper we considered the Cauchy problem \eqref{1.0.2}. The equation in \eqref{1.0.2}, for $\lambda>0$, can be seen as the dissipative version of \eqref{1.0.1}. No matter the value of $\lambda$, the problem \eqref{1.0.2} is locally well-posed as shown in Theorem \ref{teo1.1}. Similarly for the Camassa-Holm equation and other analogous equations (see \cite{const1998-1,const1998-2,silvajde2019,igor-jde,liu2011,mustafa}) \eqref{1.0.2} admits the wave breaking phenomenon (see Theorem \ref{teo1.2}), but differently from the Camassa-Holm equation, we cannot assure the global existence of solutions through the way we followed here, see \cite{hak,igor-jde,mustafa}. The only information we have about the possibility of global existence is given in Theorem \ref{teo1.2}: If the solutions has $x-$derivative bounded from below, then the solution does not blow-up in finite time. This result is a lighthouse to guide us to look for the global existence and also for the blow-up of solutions. Although, we could not determine whether this fundamental condition is satisfied, this result drove us to the direction to establish the blow-up of solutions manifested through wave breaking.

We would also like to observe that from the symmetry group of the equation in \eqref{1.0.2} we can conclude that it is only invariant with respect to translations in $t$ and $x$, which means that the most general invariant solutions of \eqref{1.0.2} are the travelling waves. Also, we note that the parameter $\lambda$ is irrelevant from the point of view of Lie symmetries.

Last, but not least, we observe that the presence of the parameter $\lambda$ makes the number of conserved quantities up to second order decreases. This can be seen by comparing the conserved quantities in Corollary \ref{cor3.1} with those given in Theorem \ref{teo3.3}. This is equivalent to say that the characteristic of the conservation laws for the equation in \eqref{1.0.2} is, actually, reduced to zero order differential functions.

\section*{Acknowledgements}

I. L. Freire is thankful to CNPq (grants 308516/2016-8 and 404912/2016-8) for financial support. N. Sales Filho thanks FaEng/UFMT for leaving to develop his PhD thesis. L. C. Souza and C. E. Toffoli are grateful to IFSP for all support provided. I. L. Freire would like to thank Professor A. Cheviakov for sharing his package GEM for calculating conserved currents and Dr. P. L. da Silva for discussions and suggestions concerning the manuscript.


\begin{thebibliography}{10}

\bibitem{anco} S. Anco, P. L. da Silva and I. L. Freire, \newblock A family of wave breaking equations generalizing the Camassa-Holm and Novikov equations, \newblock \emph{J. Math. Phys.}, vol. 56, paper 091506, (2015).

\bibitem{bk} G. W. Bluman  and S. Kumei, Symmetries and Differential Equations,  Applied Mathematical Sciences 81, Springer, New York, (1989).




\bibitem{chprl} R. Camassa, D.D. Holm, An integrable shallow water equation with peaked solitons, Phys. Rev. Lett., vol. 71, 1661--1664, (1993).


\bibitem{chen2020} A. Chen and X. Lu, Orbital stability of elliptic periodic peakons for the modified Camassa-Holm equation, Disc. Cont. Dyn. Sys., vol. 40, (2020).


\bibitem{chev-2007} A. Cheviakov, GeM software package for computation of symmetries and conservation laws of differential equations, Comp. Phys. Comm., vol. 176, 48--61, (2007).

\bibitem{chev-2010-1} A. Cheviakov, Symbolic computation of local symmetries of nonlinear and linear partial and ordinary differential equations, Math. Comput. Sci., vol. 4, 203--222, (2010).

\bibitem{chev-2010-2} A. Cheviakov, Computation of fluxes of conservation laws, J. Eng. Math., vol. 66, 153--173, (2010).

\bibitem{chev-2014} A. Cheviakov, Symbolic computation of nonlocal symmetries and nonlocal conservation laws of partial differential equations using the GeM package for Maple, Similarity and Symmetry Methods. Lecture Notes 165 in Applied and Computational Mechanics 73, Springer, (2014).

\bibitem{chev-2017} A. Cheviakov, Symbolic computation of equivalence transformations and parameter reduction for nonlinear physical models, Comp. Phys. Comm., vol. 220, 56--73, (2017).


\bibitem{clark} P. A. Clarkson, E. L. Mansfield and T. J. Priestley, Symmetries of a class of nonlinear third-order partial differential equations, Math. Comput. Modelling., 25, 195--212, (1997).


\bibitem{const1998-1} A. Constantin, J. Escher, Global existence and blow-up for a shallow water equation, Annali Sc. Norm. Sup. Pisa, vol. 26, 303--328, (1998).


\bibitem{const1998-2} A. Constantin and J. Escher, Wave breaking for nonlinear nonlocal shallow water equations, Acta Math., vol. 181, 229--243 (1998).

\bibitem{const1998-3} A. Constantin and J. Escher, Well-Posedness, Global Existence, and Blow up Phenomena, for a Periodic Quasi-Linear Hyperbolic Equation, Commun. Pure App. Math., Vol. LI, 0475--0504 (1998).

\bibitem{const2000-1} A. Constantin, Existence of permanent and breaking waves for a shallow water equation: a geometric approach, Ann. Inst. Fourier, vol. 50, 321--362, (2000).

\bibitem{const2000-2} A. Constantin and W. Strauss, Stability of peakons, Commun. Pure Appl. Math., vol. 53,
603--610, (2000).

\bibitem{const2002} A. Constantin and W. Strauss, Stability of the Camassa-Holm solitons, J. Nonlinear Sci., vol. 12, 415--422 (2002).

\bibitem{const-mol}A. Constantin, L. Molinet, The initial value problem for a generalized Boussinesq equation, Differential Integral Equations, vol. 15, 1061--1072,  (2002).

\bibitem{pri-book} P. L. da Silva and I. L. Freire, On the group analysis of a modified Novikov equation, in Interdisciplinary Topics in Applied Mathematics, Modeling and Computational Science, Springer Proceedings in Mathematics $\&$ Statistics 117, (2015), DOI 10.1007/978-3-319-12307-3$\_$23.

\bibitem{pri-aims} P. L. da Silva and I. L. Freire, An equation unifying both Camassa-Holm and Novikov equations, Discreted and Continuous Dynamical Systems, 304--311, (2015), DOI: 10.3934/proc.2015.0304.

\bibitem{raspa} P. L. da Silva, Classification of bounded travelling wave solutions for the Dullin--Gottwald--Holm equation, J. Math. Anal. Appl., vol. 471, 481--488, (2019), doi: 10.1016/j.jmaa.2018.10.086.

\bibitem{silvajde2019} P. L. da Silva and I. L. Freire, Well-posedness, travelling waves and geometrical aspects of generalizations of the Camassa-Holm equation, J. Diff. Equ., vol. 267, 5318--5369, (2019).

\bibitem{silva2019} P. L. da Silva and I. L. Freire, Integrability, existence of global solutions and blow-up criteria for a generalization of the Camassa-Holm equation, arXiv:1906.00304, (2019).

\bibitem{daros} A. Dar\'os, Estabilidade Orbital de Standing Waves, PhD Thesis, UFSCar, (2018).


\bibitem{kelly} A. Dar\'os and L. K. Arruda, On the instability of elliptic travelling wave solutions of the modified Camassa--Holm equation, J. Diff. Equ., vol. 266, 1946--1968, (2018), DOI: 10.1016/j.jde.2018.08.017.



\bibitem{escher} J. Escher, Breaking water waves, In: Constantin A. (eds) Nonlinear Water Waves. Lecture Notes in Mathematics, vol 2158. Springer, Cham, (2016), DOI: 10.1007/978-3-319-31462-4$\_$2.

\bibitem{igor-jde} I. L. Freire, Wave breaking for shallow water models with time decaying solutions, J. Diff. Equ., (2020), DOI: 10.1016/j.jde.2020.03.011.

\bibitem{igor-cm} I. L. Freire, A look on some results about Camassa--Holm type equations, to appear Communications in Mathematics, (2020).

\bibitem{hak} S. Hakkaev, I. D. Iliev and K. Kirchev, Stability of periodic traveling shallow-water waves determined by Newton's equation, J. Phys. A: Math. Theor., vol. 41, paper 085203, (2008).



\bibitem{ibrabook} N. H. Ibragimov, Elementary Lie group analysis and ordinary differential equations,
John Wiley and Sons, United Kingdom, (1999).

\bibitem{kato} T. Kato, Quasi-linear equations of evolution, with applications to partial differential equations. in:
Spectral theory and differential equations, Proceedings of the Symposium Dundee, 1974, dedicated to
Konrad Jrgens, Lecture Notes in Math, Vol. 448, Springer, Berlin, 1975, pp. 25--70.

\bibitem{lenjde2013} J. Lenells and M. Wunsch, On the weakly dissipative Camassa--Holm,
Degasperis--Procesi, and Novikov equations, J. Diff. Equ., vol. 255, 441-448, (2013).

\bibitem{linares} F. Linares and G. Ponce, Introduction to Nonlinear Dispersive Equations, Springer, (2015).

\bibitem{liu2011} X. Liu and Z. Yin, Local well-posedness and stability of peakons for a generalized Dullin--Gottwald--Holm equation, Nonlin. Anal., vol. 74, 2497--2507, (2011).

\bibitem{natali} R. H. Martins and F. Natali, A comment about the paper On the instability of elliptic traveling wave solutions of the modified Camassa-Holm equation, arXiv:1912.06599, (2019).

\bibitem{mustafa} O. G. Mustafa, On the Cauchy problem for a generalized
Camassa--Holm equation, Nonlin. Anal., vol. 64, 1382--1399, (2006).

\bibitem{olverbook}  P. J. Olver, Applications of Lie groups to differential equations, 2nd edition, Springer, New York, (1993).

\bibitem{blanco} G. Rodriguez-Blanco, On the Cauchy problem for the Camassa--Holm equation, Nonlinear Anal., 46, 309--327 (2001).

\bibitem{taylor} M. E. Taylor, Partial Differential Equations I, 2nd edition, Springer, (2011).

\bibitem{yin2010} J. Yin, L. Tian and X. Fan, Stability of negative solitary waves for an integrable modified Camassa--Holm equation, J. Math. Phys., vol. 51, paper 053515, (2010).

\bibitem{waz} A.M. Wazwaz, Solitary wave solutions for modified forms of Degasperis--Procesi and Camassa--Holm equations, Phys. Lett. A, vol. 352, 500--504, (2006).


\end{thebibliography}
\end{document}